\documentclass[a4paper, 12pt]{article}
\usepackage{verbatim}
\usepackage{amsmath,amssymb}
\usepackage{amsthm}
\usepackage[dvipdfmx]{graphicx}
\usepackage{mathrsfs}
\usepackage[active]{srcltx}
\usepackage{amscd}
\usepackage{cases}
\usepackage[noadjust]{cite}
\usepackage{float}
\usepackage[active]{srcltx}
\usepackage{array}
\newcolumntype{L}{>{\displaystyle}l}
\newcolumntype{C}{>{\displaystyle}c}
\newcolumntype{R}{>{\displaystyle}r}
 \newtheorem{theorem}{Theorem}[section]
 \newtheorem{proposition}[theorem]{Proposition}
 \newtheorem{fact}[theorem]{Fact}
 \newtheorem{lemma}[theorem]{Lemma}
 \newtheorem{corollary}[theorem]{Corollary}
\theoremstyle{definition}
 \newtheorem{definition}[theorem]{Definition}
 \newtheorem{remark}[theorem]{Remark}
 \newtheorem*{acknowledgements}{Acknowledgements}
 \newtheorem{example}[theorem]{Example}
\numberwithin{equation}{section}

\usepackage[usenames]{color}

\newcommand{\R}{\boldsymbol{R}}
\newcommand{\Z}{\boldsymbol{Z}}

\newcommand{\rank}{\operatorname{rank}}

\renewcommand{\phi}{\varphi}

\newcommand{\A}{\mathcal{A}}

\newcommand{\RR}{\mathcal{R}}

\newcommand{\ep}{\varepsilon}

\newcommand{\mycomment}[1]{}
\renewcommand{\tilde}{\widetilde}
\title{{\normalsize {\bf Geometry on deformations of $S_1$ singularities}}}
\author{{\normalsize Runa Shimada}}
\begin{document}
\maketitle
\begin{abstract}
To study a one parameter deformation of an $S_1$ singularity
taking into
consideration its differential geometric properties,
we give a form representing the deformation
using only diffeomorphisms on the source and isometries of the target.
Using this form, we study differential geometric properties of $S_1$ singularities and the Whitney umbrellas appearing in the deformation. 
\end{abstract}

\section{Introduction}

In recent decades, differential geometric properties of curves and surfaces with singularities have been studied by many authors. 
Usually, singularities are deformed and turn into various other singularities.
 The $S_1^\pm$ singularities are known as codimension one singularities of surfaces, and the codimension one singularities correspond to the appearance/disappearance of generic singularities. Therefore, it is natural to include a deformation when we study such singularities. In other words, the appearance/disappearance of the Whitney umbrella can be seen in deformations of the $S_1$ singularities.
An $SO(3)$-{\it normal form} (a {\it normal form} for short)
is a formula for a singular point  which reduces coefficients
as much as possible by using a diffeomorphism-germ on
the source space and an isometry-germ on the target space.
This kind of form is given in \cite{bw,west} for the
Whitney umbrella, and is called the Bruce-West normal form.
There are many studies of the differential geometry of the Whitney umbrella by using the Bruce-West normal form. 
 In {\rm \cite{hhnuy}}, it is proved that some of the coefficients of the form are intrinsic invariants.
Moreover, the focal conic, which measures the contact to the sphere, has been introduced using coefficients appearing in this normal form (c.f. {\rm \cite{fh-fronts}}).
The curvature parabola has been introduced in {\rm \cite{mn}} and using it, the umbilic curvature is defined except for the Whitney umbrella, and the axial curvature is defined {\rm \cite{os}}.

In this paper, we extend the normal form 
to deal with singularities which are deforming.
More precisely, we give a normal form of the 
$S_1^\pm$ singularities
including a deformation parameter (Theorem \ref{thm:normal}).
In the generic deformation of the 
$S_1^\pm$ singularity, two Whitney umbrellas appear, which merge into a single
singularity and then disappears, or vise versa.

We give 
the formula for the asymptotic expansions
of geometric invariants of Whitney umbrellas
appearing in the deformations of 
the $S_1^\pm$ singularity (Theorem \ref{thm:invexp}).
Using the formula, we study the asymptotic behavior
of the focal conic,
curvature parabola and geometric invariants of the
Whitney umbrellas appearing in the deformation
of an $S_1^\pm$ singularity in Section \ref{sec:geom}.
As a by-product, we obtain a reasonable
definition of the umbilic curvature for a Whitney umbrella
which is not defined for this singularity in {\rm \cite{mn}}.

The {\it Whitney umbrella} is a map-germ  $\A$-equivalent to the map-germ defined by
$
(u,v)\mapsto(u,uv,v^2)
$
at the origin.
Here two map-germs $f_1:(\R^2,0)\to(\R^3,0)$ and $f_2:(\R^2,0)\to(\R^3,0)$ are {\it $\A$-equivalent} if there exists a coordinate change $\phi$ of the source space and a coordinate change $\psi$ of the target space such that $f_2=\psi\circ f_1\circ \phi^{-1}$.

The following fact is known.
\begin{fact}{\rm \cite{west}}\label{fact:wumb}
Let\/ $f:(\R^2,0) \to (\R^3,0)$ be a Whitney umbrella. Then there exist a diffeomorphism-germ\/ $\phi:(\R^2,0) \to (\R^3,0)$ and\/ $T \in SO(3)$ such that\/
$$
T \circ f \circ \phi (u,v) =\biggl (u, uv+O(3),\frac{1}{2}(a_{20}u^2+2a_{11}uv+a_{02}v^2)+O(3)\biggr)
$$
with\/ $a_{20}, a_{11}, a_{02}  \in \R$. Here,\/  $O(n)$ stands for the terms whose degrees are equal to or greater than\/ $n$.
\end{fact}

See \cite{bk,diastari,gsg,hhnsuy,hhnuy,hnuy,sym,tari} for other studies using the normal form. The coefficients
$a_{20}, a_{11}, a_{02}$ 
contain
all the local second order geometrical information
about the Whitney umbrella.
An $S_1^\pm$ singularity is map-germ  $\A$-equivalent to the map-germ defined by 
$
(u, v) \mapsto (u, v^2, v(u^2\pm v^2))
$
at the origin. See  the center graphics of Figures 1 and 2.
It is known that the codimension one singularities of surfaces in 3-space are $S_1^\pm$ singularities by Mond's classification (c.f. {\rm \cite{mond}},
see also \cite{chenmatu}).
\section{Geometric deformations of  $S_1^\pm$ singularities and their normal forms}

The main objective of this paper is to study geometric deformations of singularities.

\begin{definition}
A map-germ $f:(\R^2 \times \R,0) \to (\R^3,0)$ is a {\it deformation of $g:(\R^2,0) \to (\R^3,0)$} if it is smooth and $f(u, v, 0)=g(u,v)$ and $f(0,0,s)=(0,0,0)$.
\end{definition}

In this definition, the parameter $s$ as the third component of the source space is called the {\it deformation parameter}. We define an equivalence relation between two deformations preserving the deformation parameters as follows.

\begin{definition}\label{def:deformeq}
Let $f_1, f_2 : (\R^2 \times \R, 0) \to (\R^3, 0)$ be deformations of $g$. Then $f_1$ and
$f_2$ are {\it equivalent as deformations} if there exist orientation preserving diffeomorphism-germs $\phi:(\R^2 \times \R,0) \to (\R^2 \times \R,0)$
with the form
\begin{eqnarray}
\phi(u,v,s)=(\phi_1(u,v,s),\phi_2(u,v,s),\phi_3(s))\quad
\left(
\dfrac{d\phi_3}{ds}(0)>0\right)
\label{eq:phi}
\end{eqnarray}
and 
$\psi:(\R^3 ,0) \to (\R^3,0)$
such that
$\psi \circ f_1 \circ \phi^{-1}(u,v,s)=f_2(u,v,s)$ holds.
\end{definition}
This implies that 
equivalence as deformations means an $\A$-equivalence that preserves
the deformation parameters. This notion is also called $P$-$\RR$-equivalence  (c.f. {\rm \cite{irrt}}). The following maps are typical deformations of the $S_1^\pm$ singularities.

Let $f^{s,\pm}$ be germs defined by
$$f^{s,\pm}:(\R^2 \times \R,0) \ni (u,v,s) \mapsto (u,v^2,v(u^2\pm v^2)+sv) \in (\R^3,0).$$
Then $f^{s,+} (f^{s,-})$ is a deformation of an $S_1^+ (S_1^-) $ singularity.
We can observe that the Whitney umbrellas appear in the deformations.
See Figure \ref{fig:IMG_91} and Figure \ref{fig:IMG_85}.

\begin{figure}[h!]
\centering
\includegraphics[width=0.5\linewidth, bb=0 0 892 374]{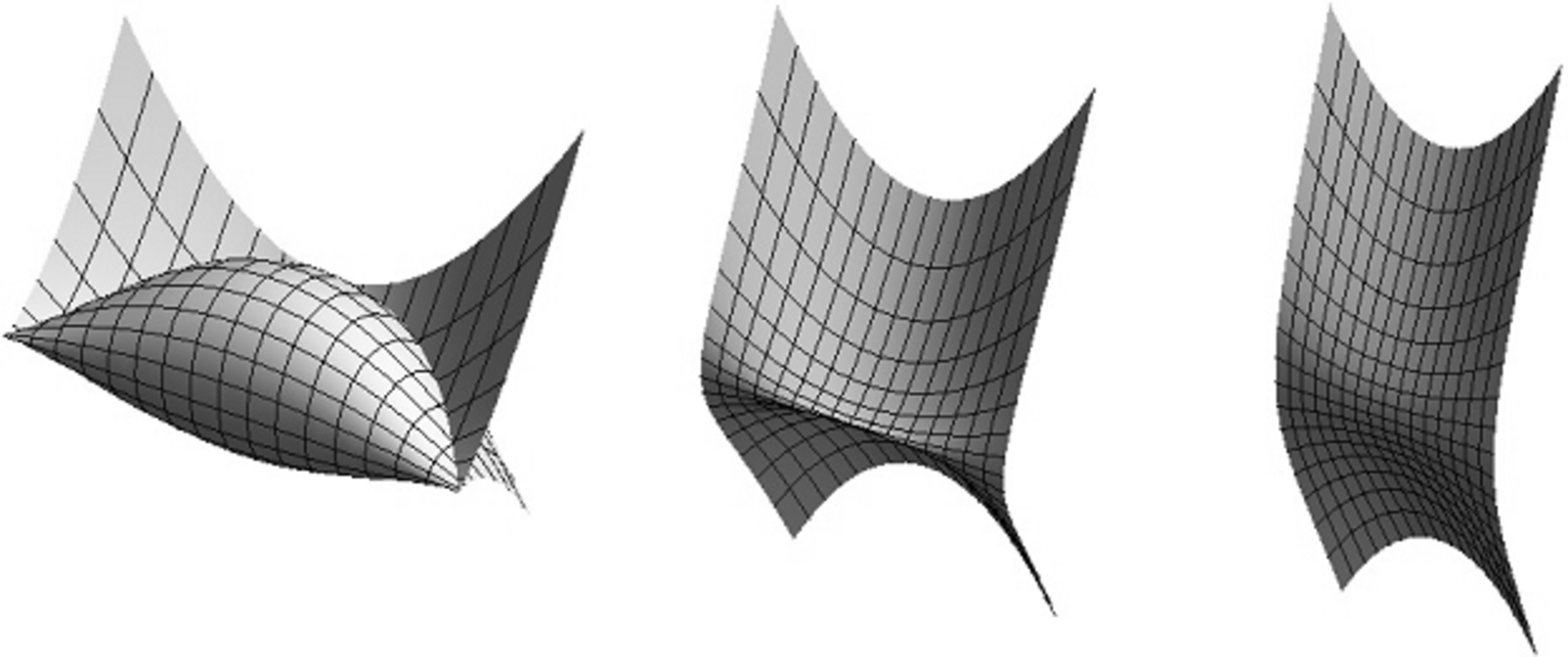}

\caption{Deformation of an $S_1^+$ singularity $($from left to right $f^{-1,+}, f^{0,+}$ and $f^{1,+})$}
\label{fig:IMG_91}
\end{figure}

\begin{figure}[htbp]
\centering
\includegraphics[width=0.5\linewidth, bb=0 0 985 368]{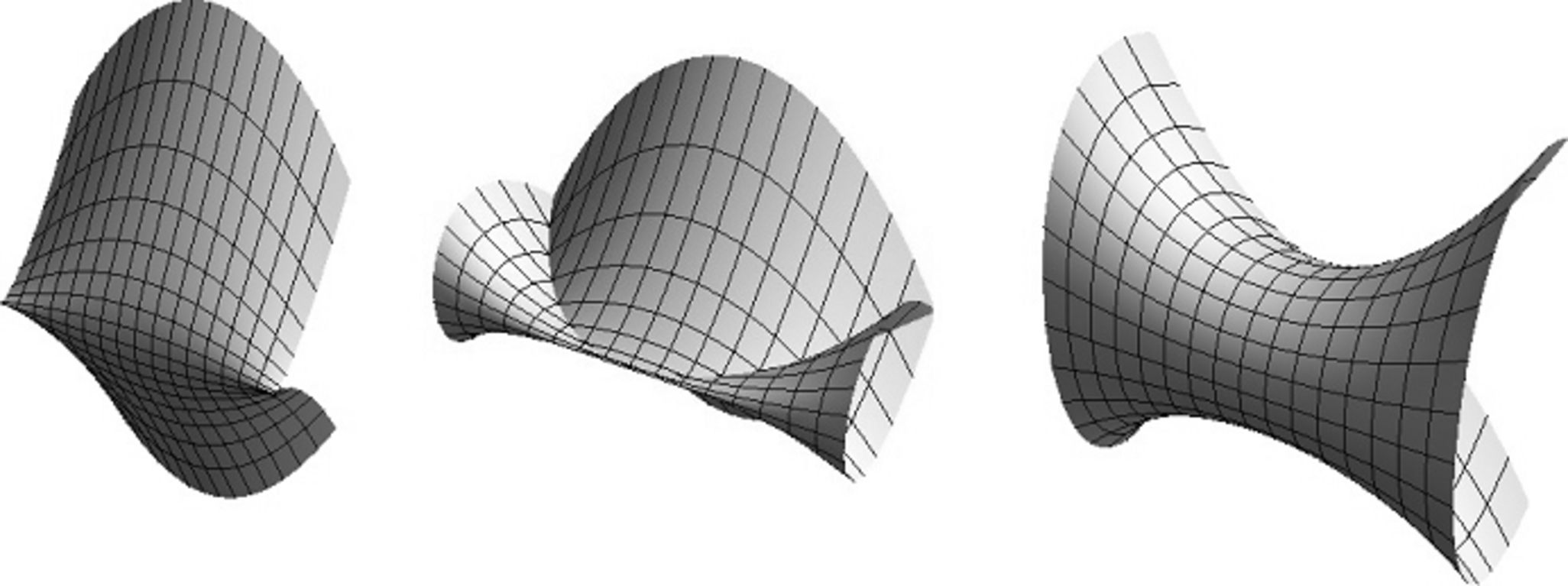}
\caption{Deformation of an $S_1^-$ singularity $($from left to right $f^{-1,-},  f^{0,-}$ and $f^{1,-}$)}
\label{fig:IMG_85}
\end{figure}

Let $ f : (\R^2, 0) \to (\R^n, 0)$ $(n=2,3)$ be a map-germ.
A vector $\eta\in T_p\R^2$ is called a {\it null vector} if
it is a generator of $\operatorname{Ker} df_p$.
A vector field $\eta$ on the source space  is an
{\it extended null vector field} if when $p\in S(f)$ then $\eta_p$ is a null vector.
An extended null vector field is also called 
a {\it null vector field} for short.

\begin{theorem}\label{thm:normal}
Let\/ $ f : (\R^2 \times \R, 0) \to (\R^3, 0)$ be a deformation of\/ 
$g:(\R^2, 0) \to (\R^3, 0)$ such that
the\/ $2$-jet of\/ $g$ is\/ $\A$-equivalent to\/ $(u,v^2,0)$.
Then there exist an orientation preserving diffeomorphism-germ\/ $\phi : (\R^2 \times \R, 0) \to (\R^2 \times \R, 0)$ with the form \eqref{eq:phi},\/ $T \in SO(3)$ and functions\/ $f_{21}, f_{31} \in C^\infty(1,1), f_{24}, f_{33}, f_{34} \in C^\infty(2,1), f_{32} \in C^\infty(3,1)$ such that
\begin{eqnarray}
{f_{{\rm normal}}}^s&=&
T \circ f \circ \varphi(u,v,s)\nonumber\\ 
&=&(u,u^2f_{21}(u)+v^2+usf_{24}(u,s),\label{eq:normal}\\
& &\hspace{5mm}
u^2f_{31}(u)+v^2f_{32}(u,v,s)+vf_{33}(u,s)+usf_{34}(u,s)),\nonumber
\end{eqnarray}\nonumber
where\/
$f_{32}(0,0,0)=f_{33}(0,0)=(f_{33})_u(0,0)=0$.
Furthermore,\/ $f(u,v,0)=g(u,v)$ is an\/ $S_1^+$ singularity
\/$($respectively, \/$S_1^-$ singularity\/$)$
if and only if\/
$$(f_{32})_v(0,0,0)(f_{33})_{uu}(0,0)> 0$$\/$($respectively, $<0)$.
If $(df_{33}/ds)(0,0) \ne 0$, then one can further reduce $f_{33}(0,s)=s$.
\end{theorem}
A uniqueness of the form \eqref{eq:normal} will be discussed in
Proposition \ref{prop:unique} later.
\begin{proof}
Let $ f : (\R^2 \times \R, 0) \to (\R^3, 0)$ satisfy the assumption.
Since $\rank df=1$, we can assume 
$$
f(u,v, s)=(u, f_2(u,v,s), f_3(u,v,s))
$$
by an orientation preserving diffeomorphism-germ 
satisfying \eqref{eq:phi}
on the source space.
The functions $f_2(u,v,s)$ and $f_3(u,v,s)$ can be written as
$$
f_2(u,v,s)=f_{21}(u)+v^2f_{22}(u,v,s)+vf_{23}(u,s)+sf_{24}(u,s),
$$
$$
f_3(u,v,s)=f_{31}(u)+v^2f_{32}(u,v,s)+vf_{33}(u,s)+sf_{34}(u,s).
$$
By a rotation of $\R^3$, we can assume $(f_{21})_u(0)=(f_{31})_u(0)=0$.
Since $f(u,v,0)=g(u,v)$ satisfies that
$j^2g(0,0)$ is $\A$-equivalent to $(u,v^2,0)$, it holds that
$f_{22}(0,0,0)\ne0$ or $f_{32}(0,0,0)\ne0$. 
By a rotation with respect to the first axis of $\R^3$, we can assume 
$f_{22}(0,0,0)>0$ and $f_{32}(0,0,0)=0$.
We set
$$
F(u,v,s)=(u,f_2(u,v,s)).
$$
Then we see $\partial_v$ is a null vector field for $F$.
Let us set
$
\lambda=(f_2)_v.
$
Then $\lambda^{-1}(0)=S(F)$  holds. Since $f_{22}(0,0,0)\ne0$, 
it holds that $\lambda_v\ne0$.
By the implicit function theorem, there exists a function $\sigma(u,s)$ such that $\lambda(u,\sigma(u,s),s)=0$ holds.
Let us set
$$\Phi(u,v,s)=(u,v-\sigma(u,s),s),$$
and it is an orientation preserving 
diffeomorphism-germ satisfying \eqref{eq:phi}, and we set
$\Phi(u,v,s)=(\tilde u, \tilde v, \tilde s)$.
Since $\tilde v=0$ if and only if $v-\sigma(u,s)=0$, we know \{$\tilde v=0$\} is the set of singular points of $F$.
Since $\partial_{\tilde v}=\partial_v$, on \{$\tilde v=0$\},
we know $\partial_{\tilde v}$ is a null vector field of $F$.
By
$(f_2)_{\tilde v}(\tilde u,0,\tilde s)=0$ the function $f_2$ satisfies $(f_2)_{\tilde v}(\tilde u,\tilde v,\tilde s)=\tilde v g(\tilde u,\tilde v,\tilde s)$,
for a function $g(\tilde u,\tilde v,\tilde s)$.
Thus we can assume $f$ has the form
$$
(u, f_{21}(u)+v^2f_{22}(u,v,s)+sf_{24}(u,s), 
 f_{31}(u)+v^2f_{32}(u,v,s)+vf_{33}(u,s)+sf_{34}(u,s))
$$
with $f_{32}(0,0,0)=0$.
Since $f_{22}(0,0,0)>0$, we have
an orientation preserving
diffeomorphism-germ $\phi(u,v,s)=(u,v\sqrt{f_{22}(u,v,s)},s)$ satisfying \eqref{eq:phi}.
By considering $f\circ \phi^{-1}$ and by $f(0,0,s)=(0,0,0)$,
we can assume
\begin{eqnarray*}
(u,u^2f_{21}(u)+v^2+usf_{24}(u,s),
u^2f_{31}(u)+v^2f_{32}(u,v,s)+vf_{33}(u,s)+usf_{34}(u,s)),
\end{eqnarray*}
with $f_{32}(0,0,0)=0$,
where we have \eqref{eq:normal}.
Since $0 \in S(f(u,v,0))$, we have $f_{33}(0,0)=0$,
and since $j^2g(0,0)$ is $\A$-equivalent to $(u,v^2,0)$,
it holds that
$$(f_{33})_u(0,0)=0.$$
By \cite[Theorem 2.2]{saji}, we see that $f$ of the form
\eqref{eq:normal} with these conditions is
an $S_1^+$ singularity (respectively,
$S_1^-$ singularity) if and only if
$$(f_{32})_v(0,0,0)(f_{33})_{uu}(0,0)> 0$$
$($respectively, $<0)$.
If $(df_{33}/ds)(0,0) \ne 0$, 
then one can assume $f_{33}(0,s)=s$ by a change of the deformation parameter. 
\end{proof}
We remark that
since the tangent space $T_e\A(f^0)$ of $f^0$ satisfies
$$T_e\A f^0+(0,0,v)\R=C^\infty(2,1)^3$$
\cite[Proposition 4.1.16]{mond},
we see if
$(df_{33}/ds)(0,0) \ne 0$,
then $f^s$ is a generic deformation,
where $f^s={f_{{\rm normal}}}^s$.
In what follows, we assume $(df_{33}/ds)(0,0) \ne 0$ and
$f_{33}(0,s)=s$ in $f_{{\rm normal}}^s$.
The form ${f_{{\rm normal}}}^s$ 
with the condition $f_{33}(0,s)=s$ 
is called the {\it normal form of the deformations of an $S_1^\pm$ singularity}.
We see that the set of singular points $S({f_{{\rm normal}}}^s)$ of ${f_{{\rm normal}}}^s$ is
$$S({f_{{\rm normal}}}^s)=\{(u,v)\,|\, v=0, f_{33}(u,s)=0\}.$$
In Theorem \ref{thm:normal}, the given $f$ and ${f_{{\rm normal}}}^s$ with this condition are equivalent as deformations, and they
have the same differential geometric properties.
The uniqueness of the normal form holds as in the following sense.

\begin{definition}
A function $g(u,s)$ has {\it regularity of order $k$ in $s$} if
$$\partial g/\partial s(0,0)=\cdots =\partial^{k-1} g/\partial s^{k-1}(0,0)=0,\quad \partial^k g/\partial s^k(0,0)\ne0.$$
A function $g(u,s)$ has {\it regularity of finite order in $s$} if
there exists $k$ such that
$g$ has regularity of order $k$ in $s$.
\end{definition}

\begin{proposition}\label{prop:unique}
Let\/ $f_{{\rm normal}}^s:(\R^2 \times \R, 0) \to (\R^3, 0)$ be
a map-germ given by the form \eqref{eq:normal}
satisfying that\/ $f_{24}$ or\/ $f_{34}$ is regularity of finite order in\/ $s$.
If there exist an
orientation preserving diffeomorphism-germ\/ 
$\phi : (\R^2 \times \R, 0) \to (\R^2 \times \R, 0)$ 
with the form of Definition\/ {\rm\ref{def:deformeq}}
and\/ $T \in SO(3)$ such that\/
$$
T\circ {f_{{\rm normal}}}^s\circ\phi(u,v,s)={f_{{\rm normal}}}^s(u,v,s),
$$
then for any\/ $k\in \Z$, it holds that\/
$\hat{u}=u$,\/
$j^k\hat{v}(0,0,0)=j^kv(0,0,0)$ and\/
$j^k\hat{s}(0)=j^ks(0)$,
where\/
$\phi(u,v,s)=(\hat{u}(u,v,s),\hat{v}(u,v,s),\hat{s}(s))$.
\end{proposition}
\begin{proof}
We set ${f_{{\rm normal}}}^s(u,v,s)=f(u,v,s)$.
Since the image of $df_{(0,0,0)}$ is
generated by $(1,0,0)$.
The kernel of $df_{(0,0,0)}$ is generated by
$\partial_v$,
and 
$(1,0,0)\times f_{vv}(0,0,0)=(0,0,1)$
holds.
Since the subspaces generated by
$(1,0,0)$ and $(0,0,1)$
are independent of the coordinate system on
the source space,
it holds that $T$ is the identity.
Therefore $\hat u=u$ holds.
Looking at the second component,
$\hat v^2+u\hat sf_{24}(u,\hat s)=v^2+usf_{24}(u,s)$
holds.
By looking at the terms $v$ and $v^{k+1}$, together with
$\phi$ being orientation preserving,
we see $\hat{v}$ has the form
\begin{equation}\label{eq:uniqprf100}
\hat{v}=v+g_0(u,s)+\sum_{i=k}^r v^ig_i(u,s)+O_v(r+1)\quad(k\geq2)
\end{equation}
with $k=2$,
where $O_v(r+1)$ stands for functions $g(u,v,s)$ satisfying
$g(u,0,s)=g_v(u,0,s)=\cdots=
\partial^r g/\partial v^r(u,0,s)=0$.
We assume \eqref{eq:uniqprf100} holds for some $k$.
Then by $\hat v^2+u\hat sf_{24}(u,\hat s)=v^2+usf_{24}(u,s)$,
\begin{align*}
\hat v^2 +u\hat sf_{24}(u,\hat s)&=u\hat sf_{24}(u,\hat s)+g_0^2(u,s)+2vg_0(u,s)+v^2+2v^kg_0(u,s)g_k(u,s)\\
&+2v^{k+1}g_k(u,s)+2v^{k+1}g_0(u,s)g_{k+1}(u,s)+O_v(k+2).
\end{align*}
Comparing the terms $v$ and $v^{k+1}$ in the above equation,
$g_0(u,s) = 0$ holds, and we have
$g_k(u,s)=0$. By the induction on $k$,
the assertion for $\hat v$ holds.
If $f_{24}$ has regularity of finite order in $s$,
then there exists
an integer $k$ such that
$$
f_{24}(u,s)=s^kh_{0}(s)+uh_{1}(u,s),\quad
h_{0}(0)\ne0.
$$
We assume
$$
\hat s=s+\sum_{i=l}^r b_is^i\quad(l\geq2).
$$
Here, we set
$$h_0(s)=\sum_{j=0}^\infty h_{0j}s^j.$$
Then by $usf_{24}(u,s)=u\hat sf_{24}(u,\hat s)$  
we have
$$u\hat sf_{24}(u,\hat s)=
u\Bigl(s+\sum_{i=l}^r b_is^i\Bigr )^{k+1}
\sum_{j=0}^{\infty}\Bigl (s+\sum_{i=l}^rb_is^i\Bigr)^j h_{0j}+u^2h_1(u,\hat s)\Bigl(s+\sum_{i=l}^r b_is^i\Bigr).$$
Seeing the term of $us^{k+l}$,  the right hand side of $usf_{24}(u,s)=u\hat sf_{24}(u,\hat s)$  is
$(k+1)b_lh_{00}+h_{0l-1}$,
so we have $b_l=0$.
By  induction on $l$, we see the assertion for $\hat s$.
If $f_{34}$ has regularity of finite order in $s$,
we have $\hat u=u$ and $\hat v=v$ for some finite jet.
Substituting 
$
v=0
$
into the third component,
one can prove
the assertion for $\hat s$ in a similar way to the
case of $f_{24}$.
\end{proof}
\begin{remark}
We remark that the normal form for the $S_1^\pm$ singularity itself
is given in \cite{fh-distance} independently.
By a direct calculation,
${f_{{\rm normal}}}^s$ can be rewritten
in the form
\begin{align}
\left(u,\dfrac{v^2}{2}+\sum_{i=1}^kb_i(s)\dfrac{u^i}{i!},
\sum_{i+j=1}^ka_{ij}(s)\dfrac{u^iv^j}{i!j!}\right)\label{eq:s1form}
\end{align}
with $b_1(0)=a_{10}(0)= a_{01}(0)=a_{11}(0)=a_{02}(0)=0, a_{21}(0)\ne0$ and $a_{03}(0)\ne0$.
Here,
$$b_1(s)=sf_{24}(0,0)+O_s(2),\quad
b_2(s)=2f_{21}(0)+2s(f_{24})_u(0,0)+O_s(2),$$
$$b_3(s)=6(f_{21})_u(0)+3s(f_{24})_{uu}(0,0)+O_s(2)$$ and
$$a_{10}(s)=sf_{34}(0,0)+O_s(2),\quad
a_{01}(s)=s(f_{33})_s(0,0)/\sqrt{2}+O_s(2),$$
$$a_{20}(s)=2f_{31}(0)+2s(f_{34})_{u}(0,0)+O_s(2),\quad
a_{11}(s)=s(f_{33})_{us}(0,0)/\sqrt{2}+O_s(2),$$
$$a_{02}(s)=s(f_{32})_{s}(0,0,0)+O_s(2),$$
$$a_{30}(s)=6(f_{31})_u(0)+3s(f_{34})_{uu}(0,0)+O_s(2),$$
$$a_{21}(s)=(f_{33})_{uu}(0,0)/\sqrt{2}+s(f_{33})_{uus}(0,0)/\sqrt{2}+O_s(2),$$
$$a_{12}(s)=(f_{32})_{u}(0,0,0)+s(f_{32})_{us}(0,0,0)+O_s(2),$$
$$a_{03}(s)=3(f_{32})_{v}(0,0,0)/\sqrt{2}+3s(f_{32})_{vs}(0,0,0)/\sqrt{2}+O_s(2),$$
where $O_{s}(2)$ stands for the terms whose degrees with
respect to $s$ are greater than or equal to two.
 Changing $v\mapsto \sqrt{2}v $ in \eqref{eq:s1form}, we obtain \eqref{eq:normal}.
 Comparing the coefficients, we obtain the above results.

\end{remark}	
\section{Geometric deformations of $S_1^\pm$ singularities}\label{sec:geom}
In this section, we study differential geometric properties of 
deformations of $S_1^\pm$ singularities.
We set $f={f_{{\rm normal}}}^s$ (see \eqref{eq:normal}). 
Here, we assume $f_{33}(0,s)=s$ 
(see the remark just after the proof of Theorem \ref{thm:normal}).
\subsection{Geometric invariants of deformations}
Here we give a criterion for Whitney umbrellas.
\begin{lemma}{\rm \cite[p161, 162]{whitney}}\label{lem:wumb}
Let\/ $f:(\R^2,0)\to (\R^3,0)$ be a map-germ, and let\/ $(u,v)$ be 
a coordinate system satisfying\/
$\operatorname{Ker}df_0= \langle \partial_v\rangle$.
Then\/ $f$ is a Whitney umbrella if and only if\/
$|f_u, f_{vv}, f_{uv}|\ne0$, where\/ $|~|=\det(~)$.
\end{lemma}
The condition $S(f)  \neq \emptyset$ is equivalent to $s\le 0$.
We set $s=-\tilde{s}^2$.
If $(u, v)\in S(f)$, then $v = 0$ 
and $u$ is dependent on $\tilde s$, which we denote by $u(\tilde s)$.
Since $f_{33}(0,0)=(f_{33})_u(0,0)=0$ in \eqref{eq:normal}, we set
$$
f_{33}(u,s)=s+usc_{1}(s)+u^2c_{2}(s)+u^3c_{3}(s)+u^4c_{4}(u,s).
$$
Then we have the following theorem.
\begin{theorem}\label{thm:singcurve}
If\/ $(u,v) \in S(f)$ and\/ $c_{2}(0)>0$, then\/ $u(\tilde s)$ can be
expanded as follows\/{\rm :}
\begin{align}
u(\tilde s)
&=
\dfrac{1}{c_{20}}\tilde{s}+\dfrac{1}{2c_{20}^4}\Bigl(c_{1}(0)c_{20}^2-c_{3}(0)\Bigr)\tilde{s}^2\nonumber\\
&\hspace{5mm}+\dfrac{1}{8c_{20}^7}\Bigl(\bigl(c_{1}^2(0)+4(c_{2})_{s}(0)\bigr)c_{20}^4\label{eq:u}\\
&\hspace{5mm}-2\bigr(3c_{3}(0)c_{1}(0)+2c_{4}(0,0)\bigr)c_{20}^2+5c_{3}^2(0)\Bigr)\tilde s^3+O_{\tilde s}(4),
\nonumber
\end{align}
where\/ $c_{2}(0)=c_{20}^2$.
If\/ $\tilde s\ne 0$, then\/ $f$ at\/ $(u(\pm\tilde s),0)$ are both Whitney umbrellas.
\end{theorem}

\begin{proof}
Let us set
$$
u(\tilde s)=\alpha_1 \tilde s+\alpha_2 \tilde s^2+\alpha_3 \tilde s^3+\alpha_4(\tilde s)\tilde s^4.
$$
Differentiating $f_{33}(\alpha_1 \tilde s+\alpha_2 \tilde s^2+\alpha_3 \tilde s^3+\alpha_4(\tilde s)\tilde s^4,-\tilde s^2) =0$ three times, we have
$$
\alpha_1=\dfrac{1}{c_{20}},\quad \alpha_2=\dfrac{1}{2c_{20}^4}\Bigl(c_{1}(0)c_{20}^2-c_{3}(0)\Bigr),
$$
$$\alpha_3=\dfrac{1}{8c_{20}^7}\Bigl(\bigl(c_{1}^2(0)+4(c_{2})_{s}(0)\bigr)c_{20}^4
-2\bigr(3c_{3}(0)c_{1}(0)+2c_{4}(0,0)\bigr)c_{20}^2+5c_{3}^2(0)\Bigr).
$$
This proves the first assertion.
Since 
$\operatorname{Ker}df= \langle \partial_v \rangle$
at any point $(u(\tilde s),0)$,
by applying Lemma \ref{lem:wumb}, we obtain the assertion that
$f$ at $(u(\pm\tilde s),0)$ are both Whitney umbrellas,
proving the second assertion.
\end{proof}

Since Whitney umbrellas appear on the deformations of the $S_1^\pm$ singularities, we recall a formula for invariants of the Whitney umbrella.
 If $c_{2}(0)<0$, the same calculation can be done by setting $c_{2}(0) = -c_{20}^2$,
and we obtain the same results.
 
\begin{proposition}{\rm \cite[(12, 13, 14)]{hhnuy}}\label{prop:wuinv}
Let\/ $f:(\R^2,p) \to (\R^3,0)$ be a Whitney umbrella. For a coordinate system\/ $(u,v)$ satisfying\/ $\operatorname{Ker}df_0= \langle \partial_v\rangle$ and\/ $|f_u, f_{uv}, f_{vv}|>0$, the invariants\/ $a_{20},a_{11},a_{02}$ in Fact\/ {\rm \ref{fact:wumb}} can be written as\\\/
$$
a_{20}=\dfrac{1}{4}A^{-\frac{3}{2}}B^{\frac{1}{2}}C^{-2}D(p),\quad
a_{11}=\dfrac{1}{2}A^{-\frac{1}{2}}C^{-2}E(p),\quad
a_{02}=A^{\frac{1}{2}}B^{\frac{3}{2}}C^{-2}(p),$$
where\/
$$A=f_u\cdot f_u,\quad B=(f_u\times f_{vv})\cdot(f_u\times f_{vv}),\quad  C=|f_u, f_{uv},f_{vv}|,$$
$$\/
D=|f_u, f_{uu}, f_{vv}|^2+4|f_u, f_{uv}, f_{vv}||f_u, f_{uv}, f_{uu}|,
$$
$$\/
E=
2|f_u, f_{uv}, f_{vv}|
\begin{vmatrix}
f_u \cdot f_u & f_u \cdot f_{uv} \\
f_{vv} \cdot f_u & f_{vv} \cdot f_{uv}
\end{vmatrix}
-|f_u \times f_{vv}|^2|f_u, f_{uu}, f_{vv}|.
$$
\end{proposition}

As we saw above, $(u(\tilde s),0)$ is a
singular point of $f={f_{{\rm normal}}}^s$.
Considering the above invariants $a_{20}, a_{11}, a_{02}$ at 
$(u(\tilde s),0)$,
one can regard these invariants
as functions of $\tilde s$.
Let
$a_{20}(\tilde{s})$, $a_{11}(\tilde{s})$  $a_{02}(\tilde{s})$
denote this dependency.
We remark that
$$\lim_{\tilde s\to0}|f_u, f_{uv}, f_{vv}|(u(\tilde s),0)=0.$$
Thus these functions generally diverge at the $S_1^\pm$ singularities. More precisely, we obtain the following.

\begin{theorem}\label{thm:invexp}
The functions\/  $a_{20}(\tilde{s}), a_{11}(\tilde{s})$ and\/ $a_{02}(\tilde{s})$ can be expanded as follows:\\
\medskip\/
$a_{20}(\tilde{s})
=\dfrac{1}{\tilde{s}^2}\Bigl(\dfrac{f_{31}^2(0)}{2c_{20}^2}\\
\hspace{20mm}
-\dfrac{f_{31}(0)\bigl(f_{31}(0)c_{3}(0)-c_{20}^2(3(f_{31})_u(0)-d_{1}(0,0,0)f_{21}(0))\bigr)}{c_{20}^5}\tilde{s}
+O_{\tilde s}(2)\Bigr)$,\\
\medskip
$a_{11}(\tilde{s})=\dfrac{1}{\tilde{s}^2}\Bigl(
\dfrac{f_{31}(0)}{2c_{20}^2}-\dfrac{2c_{3}(0)f_{31}(0)-c_{20}^2(3(f_{31})_u(0)-d_{1}(0,0,0)f_{21}(0))}{2c_{20}^5}\tilde{s}+O_{\tilde s}(2)\Bigr),$\\
$a_{02}(\tilde{s})
=\dfrac{1}{\tilde{s}^2}\Bigl(\dfrac{1}{2c_{20}^2}+O_{\tilde s}(1)\Bigr)$,\\
\end{theorem}

\begin{proof}
Since  
$f_{32}(0,0)=0$, the function $f_{32}$ has the form
$$
f_{32}(u,v,s)=ud_{1}(u,v,s)+vd_{2}(u,v,s)+sd_{3}(u,v,s).
$$
Since the functions $A, B, C, D$ and $E$ depend on $\tilde s$, they can be expanded as follows:
$$A(\tilde s)=A_0+A_1\tilde s+A_2 \tilde s^2+A_3(\tilde s)\tilde s^3,\quad
B(\tilde s)=B_0+B_1\tilde sB_2 \tilde s^2+B_3(\tilde s) \tilde s^3,$$
$$C(\tilde s)=C_1\tilde s+C_2\tilde s^2+C_3(\tilde s) \tilde s^3, \quad
D(\tilde s)=D_0+D_1\tilde s+D_2 \tilde s^2+D_3(\tilde s) \tilde s^3,$$
$$E(\tilde s)=E_0+E_1\tilde s+E_2\tilde s^2+E_2(\tilde s) \tilde s^3.$$
Then it holds that
\begin{align*}
a_{20}(\tilde{s})&=
\dfrac{1}{\tilde s^2}\Bigl(\dfrac{B_0^{1/2}D_0}{4A_0^{3/2}C_1^2}\\
&\hspace{10mm}
+\dfrac{-3A_1B_0C_1D_0+A_0(B_1C_1D_0-4B_0C_2D_0+2B_0C_1D_1)}{8A_0^{5/2}B_0^{1/2}C_1^3}\tilde s+O_{\tilde s}(2)\Bigr),\\
a_{11}(\tilde{s})
&=\dfrac{1}{\tilde s^2}\Bigl(\dfrac{E_0}{2A_0^{1/2}C_1^2}+\dfrac{-A_1C_1E_0-4A_0C_2E_0+2A_0C_1E_1}{4A_0^{3/2}C_1^3}\tilde s+O_{\tilde s}(2)\Bigr),\\
a_{02}(\tilde{s})&=\dfrac{1}{\tilde s^2}\Bigl(\dfrac{A_0^{1/2}B_0^{3/2}}{C_1^2}+O_{\tilde s}(1)\Bigr).
\end{align*}
Using \eqref{eq:u}, we have
$$A_0=1,\quad A_1=0,\quad A_2=\dfrac{4(f_{21}^2(0)+f_{31}^2(0))}{c_{20}^2},\quad B_0=4,\quad B_1=0,$$
$$ B_2=\dfrac{4(d_1^2(0,0,0)+4f_{31}^2(0))}{c_{20}^2},\quad C_1=-4c_{20},\quad C_2=-\dfrac{4c_{3}(0)}{c_{20}^2},$$
$$D_0=16f_{31}^2(0), \quad D_1=\dfrac{32f_{31}(0)(3(f_{31})_u(0)-d_{1}(0,0,0)f_{21}(0))}{c_{20}},$$
\begin{align*}
D_2&=\dfrac{16}{c_{20}^4}\Bigl(2{c_{20}}^4\bigr(2f_{21}(0)
  + {f_{31}}(0)
   ({f_{21}}(0)
   {d_3}(0,0,0)-(f_{34})_{u}(0,0))\bigr)\\
   &+{c_{20}}^2
   \bigl(-{d_1}(0,0,0)
   ({c_1}(0)
   {f_{21}}(0) {f_{31}}(0)+6
   {f_{31}}(0)
   {f_{21}}'(0)+6
   {f_{21}}(0)
   {f_{31}}'(0))\\
   &-2{f_{21}}(0) {f_{31}}(0)
   (d_1)_{u}(0,0,0)+{f_{21}}^2(0)
   {d_1}^2(0,0,0)+3
   {c_1}(0) {f_{31}}(0)
   {f_{31}}'(0)\\
   &+6
   {f_{31}}(0)
   {f_{31}}''(0)+9
   ({f_{31}}'(0))^2\bigr)
   +{c_3}(0) {f_{31}}(0)
   ({f_{21}}(0)
   {d_1}(0,0,0)-3
   {f_{31}}'(0))
  \Bigr),
   \end{align*}
$$ E_0=16f_{31}(0),\quad E_1=\dfrac{16(3(f_{31})_u(0)-d_{1}(0,0,0)f_{21}(0))}{c_{20}},$$
\begin{align*}
E_2=&\dfrac{8}{c_{20}^4}\Bigl(2
   {c_{20}}^4
  ({f_{21}}(0)
   {d_3}(0,0,0)-(f_{34})_{u}(0,0))\\
   &+
   {c_{20}}^2
  \bigl({c_1}(0) (3
   {f_{31}}'(0)-{f_{21}}(0)
   {d_1}(0,0,0))\\
   &-2
   {f_{21}}(0)
   (d_1)_{u}(0,0,0)-6
   {d_1}(0,0,0)
   {f_{21}}'(0)\\
   &+2
   {f_{31}}(0)
   {d_1}^2(0,0,0)+6
   {f_{31}}''(0)+8
   {f_{31}}^3(0)\bigr)\\
   &+{c_3}(
   0) ({f_{21}}(0)
   {d_1}(0,0,0)-3
   {f_{31}}'(0))\Bigr).
   \end{align*}
Hence we get the assertion.
\end{proof}

The following corollary gives geometric meanings to the
above expressions for
 $a_{20}(\tilde{s})$, $a_{11}(\tilde{s})$, $a_{02}(\tilde{s})$.

\begin{corollary}\label{cor:infinity}
Taking the limit\/
$\tilde{s} \rightarrow 0$, then the following hold:
The condition\/ $f_{31}(0) \ne 0$ is equivalent tot\/
$a_{20}(\tilde{s})$ diverging to\/ $+\infty$, and\/
$f_{31}(0)=0$ is equivalent to\/ $a_{20}(\tilde{s})$ being bounded.
The condition\/ $f_{31}(0) \ne 0$ or\/ $2c_{3}(0)f_{31}(0)-c_{20}^2(3(f_{31})_u(0)-d_{1}(0,0,0)f_{21}(0))\ne 0$ is equivalent to\/
$a_{11}(\tilde s)$ diverging, 
and\/ $f_{31}(0)=3(f_{31})_u(0)-d_{1}(0,0,0)f_{21}(0)=0$ is equivalent to\/ $a_{11}(\tilde s)$ being bounded.
The invariant\/ $a_{02}(\tilde{s})$ diverges to\/ $+\infty$.
\end{corollary}

We remark that there is a relationship between 
this corollary and the singular curvature of cuspidal edges.
Let $f:(\R^2,0)\to(\R^3,0)$ be a front, and let $0$
be a non-degenerate singular point.
Assume that the singular points of $f$ consist of
cuspidal edges except for $0$.
Then the singular curvature diverges to $-\infty$
near $0$
(see \cite[Corollary 1.14]{front} for details).
We remark that Corollary \ref{cor:infinity} can be
interpreted as a variant of this fact
for $S_1^\pm$ singularities.
The
geometric meaning of $a_{20}a_{02}$ diverging to $+\infty$ will be
clarified in Corollary
\ref{cor:hyperbola} in Section
\ref{sec:focalconic}.

We consider the behavior of the Gaussian curvature on the
set of regular points in the deformation of
an $S_1^\pm$ singularity.

\begin{theorem}\label{thm:gauss}
Let\/ $f:(\R^2\times \R,0) \to (\R^3,0)$ be a deformation of an\/
$S_1^\pm$ singularity, given by\/ $f={f_{{\rm normal}}}^s$\/ $($see $\eqref{eq:normal})$.
We assume\/ $f_{31}(0)\ne0$.
For any\/ $\theta\in (0,\pi)\cup(\pi,2\pi)$,
we set\/
$$
R=\left\{
\begin{array}{ll}
1&\quad(c_{20}^2+3d_{2}(0,0,0)>0)\\
\sqrt{
\dfrac{
c_{20}^2}{-\sin^2\theta(c_{20}^2+3d_{2}(0,0,0))+c_{20}^2}
}
&\quad(c_{20}^2+3d_{2}(0,0,0)\leq0).
\end{array}
\right.
$$
Then there exists\/ $\tilde s_0>0$ such that
for any\/ $\tilde s$ satisfying\/ $0<|\tilde s|<\tilde s_0$,
the sign of the Gaussian curvature
coincides with
that of the product
of\/ $\tilde s$,\/ $\sin\theta$ and\/ $f_{31}(0)$
on the line segment\/
$$
\{(r\cos\theta,r\sin\theta)\,|\,
0<r<Ru(\tilde s)\},
$$
where\/ $(u(\tilde s),0)\in S({f_{{\rm normal}}}^{-\tilde s^2})$.
\end{theorem}
This theorem implies that the sign of $f_{31}(0)$
gives the sign of the Gaussian curvature
around the singular points.
\begin{proof}
Let us set
$$
L=f_{uu} \cdot (f_{u} \times f_{v}),\quad
M=f_{uv} \cdot (f_{u} \times f_{v}),\quad
N=f_{vv} \cdot (f_{u} \times f_{v})
$$
and $K=LN-M^2$ as functions of $(u,v,s)$.
Then the sign of the Gaussian curvature coincides
with that of $K$ except for the set of singular points.
We set $s=-\tilde s^2$,
and
$r=ku(\tilde s)$ $(0<k<R)$.
Then $K(r\cos \theta,r\sin\theta,-\tilde s^2)$
can be calculated as
\begin{equation}\label{eq:ktildesk}
K=\tilde s^3 \tilde K,\quad
\Bigl(\tilde K=
\dfrac{8kf_{31}(0)\sin\theta}{c_{20}^3}
(c_{20}^2-c_{20}^2k^2\cos^2\theta+3 k^2d_{2}(0,0,0)\sin^2\theta)
+O_{\tilde s}(1)\Bigr),
\end{equation}
where $O_{\tilde s}(1)$ stands for the terms whose degrees with
respect to $\tilde s$ is greater than or equal to one.
Let us set
$$
g(k)=\dfrac{8kf_{31}(0)\sin\theta}{c_{20}^3}\tilde g(k),\quad
\Big(
\tilde g(k)=c_{20}^2-c_{20}^2k^2\cos^2\theta+3 k^2d_{2}(0,0,0)\sin^2\theta
\Big).
$$
Then we have the following:
\begin{lemma}\label{lem:gkpositive}
It holds that\/ $\tilde g(k)>0$ for\/ $0<k<R$.
\end{lemma}
\begin{proof}
If $c_{20}^2+3d_{2}(0,0,0)>0$, namely $R=1$,
then $\tilde g(k)$ is a first-degree function of $k^2$.
Since $\tilde g(0)=c_{20}^2>0$ and $\tilde g(1)=\sin^2\theta(c_{20}^2+3d_{2}(0,0,0))>0$,
we see $\tilde g(k)>0$.
If $c_{20}^2+3d_{2}(0,0,0)<0$ then by the definition of $R$ and $k<R$,
$$
k^2<
\dfrac{c_{20}^2}{-\sin^2\theta(c_{20}^2+3d_{2}(0,0,0))+c_{20}^2}
$$
holds.
By a direct calculation, this is equivalent to $\tilde g(k)>0$.
Thus we have $\tilde g(k)>0$ for $0<k<R$.
\end{proof}
We return to the proof of Theorem \ref{thm:gauss}.
We fix $k$ satisfying $0<k<R$.
Since
the function $\tilde K$ is continuous,
for any $\ep>0$, there exists $\delta>0$ such that
whenever $|\tilde s|<\delta$, we have
$$
|\tilde K(k,\tilde s)-\tilde K(k,0)|<\ep.
$$
Since
$
\tilde K(k,\tilde s)-\tilde K(k,0)
=
\tilde K(k,\tilde s)
-
g(k),
$
and $g(k)$ is a non-zero constant,
by taking $\ep$ smaller than
$|g(k)|$,
then whenever $|\tilde s|<\delta$,
the sign of $\tilde K(k,\tilde s)$ and
that of $\tilde K(k,0)$ coincide.
By \eqref{eq:ktildesk}, the sign of $\tilde K(k,0)$ coincides with
that of the product of $\sin\theta$, $f_{31}(0)$, $\tilde g(k)(>0)$.
Hence by $K=\tilde s^3 \tilde K$,
we have the assertion.
\end{proof}

Let us set $f^\pm(u,v,s)=(u,v^2+us, \pm u^2+v^3+ u^2v+vs)$. 
Then $f^+$ is the case of $f_{31}(0)>0$, and $f^-$ is the case of $f_{31}(0)<0$.
The figures of $f^\pm$
with $s=-1/2$ are
given in Figure \ref{fig:gauss_plus}. On both surfaces,
one can observe that
the Gaussian curvature changes sign
across the line connecting the two singular points.
On the other hand, 
as $s$ changes from negative to positive,
the regions of positive and negative Gaussian curvature 
of $f^\pm$ are interchanged.

\begin{figure}[htbp]
\centering

\includegraphics[width=0.4\linewidth, bb=0 0 495 374]{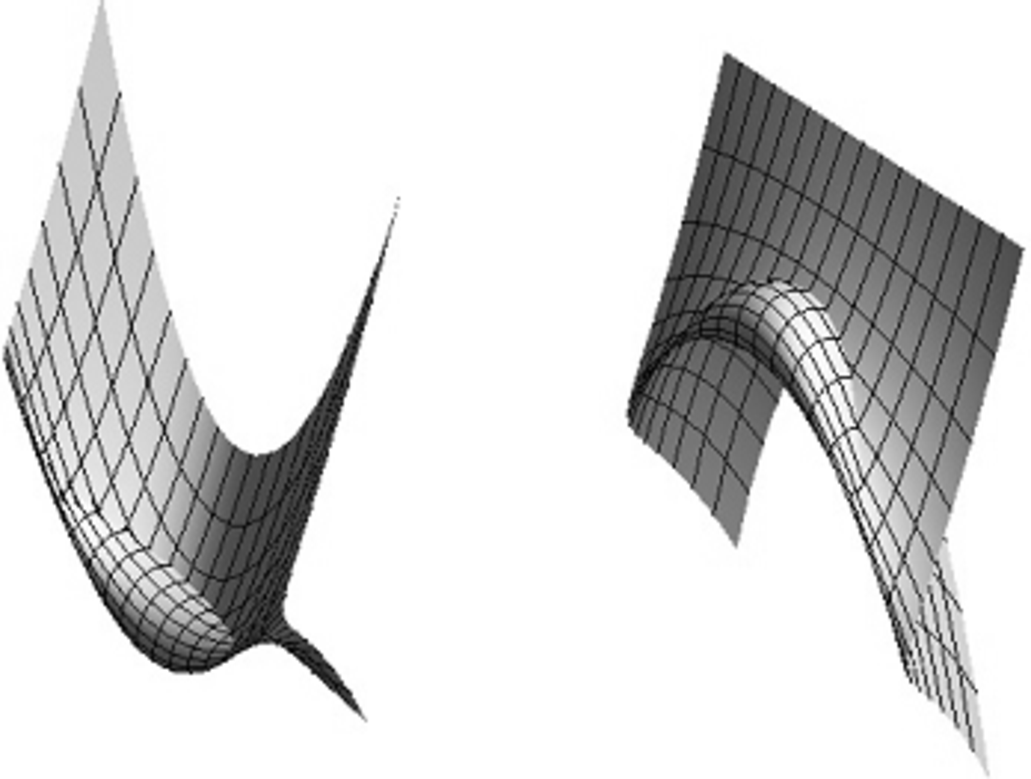}

\caption{The surfaces $f^+$ (left) and $f^-$ (right).}
\label{fig:gauss_plus}.
\end{figure}

We remark that there also is a relationship between
this corollary and the Gaussian curvature of cuspidal edges.
Let $f:(\R^2,0)\to(\R^3,0)$ be a front, and let $0$
be a non-degenerate singular point.
Assume that the singular points of $f$ consist of
cuspidal edges except for $0$.
If the normal curvature does not vanish,
then the Gaussian curvature diverges
along the set of singular points, and
it diverges to $+\infty$ or $-\infty$
across the set of singular points \cite[Corollary 3.6]{front}.
We remark that Theorem \ref{thm:gauss} can be
interpreted as a variant of this fact
for Whitney umbrellas.
The invariant $f_{31}(0)$ can be interpreted as the
umbilic curvature and it relates with the normal curvature.

\subsection{Curvature parabolas}
In this section, we consider 
geometric deformation of the curvature parabola.
It is known that the curvature parabola of 
a Whitney umbrella is a parabola,
and that of an $S_1^\pm$ singularity is a half-line.
We investigate 
how two parabolas appear from a half-line
as curvature parabolas associated with a Whitney umbrella 
and an $S_1^\pm$ singularity, respectively.

Following \cite{mn}, we introduce the curvature parabola and umbilic curvature. 
See \cite{mn} for details.
Let $ f : (\R^2,(u_0,v_0)) \to \R^3$ be a map-germ with
$\rank df_{(u_0,v_0)}=1$, and let $(u,v)$ be a
coordinate system.
Let $X$ be a vector field on $(\R^2,(u_0,v_0))$.
We set
$$
I(X,X)=df_{(u_0,v_0)}(X)\cdot df_{(u_0,v_0)}(X).
$$
Namely, if $X=a\partial_u+b\partial_v$, then
$I(X,X)=a^2E+2abF+b^2G$, where $E=f_u\cdot f_u$,
$F=f_u\cdot f_v$, $G=f_v\cdot f_v$.
Let $l$ be the line of the image of $df_{(u_0,v_0)}$,
and let $\pi$ be the orthonormal projection 
$\pi:\R^3\to l^\perp$.
For a vector field
$X=a\partial_u+b\partial_v$, we set
$$
II(X,X)=
\pi(a^2 f_{uu}(u_0,v_0)+2ab f_{uv}(u_0,v_0)+b^2f_{vv}(u_0,v_0)).
$$
Under these settings, the 
{\it curvature parabola\/} at $(u_0,v_0)$ is defined by
$$
C_{(u_0,v_0)}=\{II(X,X)\,|\,I(X,X)=1\}.
$$
If $f$ at $(u_0,v_0)$ is a Whitney umbrella,
then $C_{(u_0,v_0)}$ is a parabola, and if
$f$ at $(u_0,v_0)$ is an $S_1^\pm$ singularity,
then $C_{(u_0,v_0)}$ is a half-line.
We assume that $C_{(u_0,v_0)}$ is a half-line, and
we set $\nu_2$ to be a unit directional vector of the half-line.
Then $II(X,X)\cdot\nu_2$ does not depend on the choice of
$X$ satisfying $I(X,X)=1$.
The {\it umbilic curvature} $\kappa_u$ is
defined by $\kappa_u=|II(X,X)\cdot\nu_2|$.
The umbilic curvature 
is not defined for Whitney umbrellas.

Let $f_w(u,v)=(u,uv,(a_{20} u^2+2a_{11}uv
+a_{02}v^2)/2)$ be a Whitney umbrella.
Then 
$C_{(0,0)}=\{(2b,b^2 a_{02}+2ba_{11}+a_{20})\,|\,b\in\R\}$
holds, and it is a parabola with
the vertex $(-2a_{11}/a_{02},(a_{20}a_{02}-a_{11}^2)/a_{02})$, and
the direction of axis $(0,1)$
in the $yz$-plane in $\R^3$.
On the other hand, let us set
$f={f_{{\rm normal}}}^0$.
Then 
$C_{(0,0)}=\{2(f_{21}(0),f_{31}(0))+2(b^2,0)\,|\,b\in\R\}$
holds and it is a half-line with
the vertex $2(f_{21}(0),f_{31}(0))$ and
the direction of axis $(1,0)$
in the $yz$-plane in $\R^3$.
By Theorem \ref{thm:invexp}, we have
$$
\lim_{\tilde s\to0}\dfrac{2a_{11}(\tilde s)}{a_{02}(\tilde s)}
=2f_{31}(0).
$$
Thus the umbilic curvature for a Whitney umbrella
should be defined to be $2|a_{11}/a_{02}|$
as the projection of the vertex 
with respect to the direction that is normal to the axis.
This implies that we can extend the definition
of umbilic curvature to Whitney umbrellas in the deformation,
beyond the initial $S_1^\pm$ singularity, that is,
to other values of $\tilde s$.
In fact, in the above definition,
the umbilic curvature
as a function of $\tilde s$ is continuous.

On the other hand,
the {\it axial curvature} $\kappa_a$, which is defined in {\cite{os}}, is
$(a_{20}a_{02}-a_{11}^2)/a_{02}$, and
for $f={f_{{\rm normal}}}^0$, it is
$\kappa_a=2f_{21}(0)$.
By  Proposition \ref{prop:wuinv} and Theorem \ref{thm:invexp},
we have
\begin{align*}
\lim_{\tilde s\to0}\dfrac{a_{20}(\tilde s)a_{02}(\tilde s)-a_{11}^2(\tilde s)}{a_{02}(\tilde s)}
&=\lim_{\tilde s\to0}\dfrac{A^{-1}C^{-4}(B^2D-E^2)/4}{A^{1/2}B^{3/2}C^{-2}}\\
&=\lim_{\tilde s\to0}\dfrac{\dfrac{f_{21}(0)}{c_{20}^2}+O_{\tilde s}(1)}{\dfrac{1}{2c_{20}^2}+O_{\tilde s}(1)}\\
&=2f_{21}(0),
\end{align*}
which implies the axial curvature is also continuous
as a function of $\tilde s$.

\subsection{Focal conics}\label{sec:focalconic}
In this section, we consider 
geometric deformation of the focal set.
It is known that the focal set of a Whitney umbrella is a conic,
and that of an $S_1^\pm$ singularity consists of two lines.
We investigate 
how two conics appear from two lines
as focal sets associated with a Whitney umbrella 
and an $S_1^\pm$ singularity, respectively.

\begin{definition}
Let\/ $ f : (\R^2,(u_0,v_0)) \to \R^3$ be a map-germ with\/
$\rank df_{(u_0,v_0)}=1$, and let\/ $(u,v)$ be a
coordinate system.
The {\it focal set} of\/ $f$ at\/ $(u_0,v_0)$ is defined by\/
$$
\{x\in \R^3|D_u^x(u_0,v_0)=D_v^x(u_0,v_0)=0,
D_{uu}^x(u_0,v_0)D_{vv}^x(u_0,v_0)-D_{uv}^x(u_0,v_0)^2=0\},
$$
where\/ $D^x$ is the distance squared function\/
$D^x(u,v)=\frac{1}{2}|x-f(u,v)|^2$ $(x\in\R^3)$.
\end{definition}
Since the focal set is a conic for the Whitney umbrella  \cite{fh-fronts}, 
the focal set is called the {\it focal conic}.
The classification for the focal conic
of a Whitney umbrella is known as follows:
\begin{proposition}\label{prop:focalconic}{\rm \cite[Proposition 3.4]{fh-fronts}}
Let\/ $f:(\R^2,0) \to (\R^3,0)$ be a map-germ with 
a Whitney umbrella. Then the focal conic of\/ $f$ at\/ $0$ is
an ellipse if and only if\/ $a_{02}a_{20}<0$,
a hyperbola if and only if\/  $a_{02}a_{20}>0$, and
a parabola if and only if\/ $a_{20}=0$.
\end{proposition}

We have the following corollary for focal conics
appearing on deformations of  $S_1^\pm$ singularities:
\begin{corollary}\label{cor:hyperbola}
In the deformations of\/ ${S_1}^\pm$ singularities 
with\/ $f_{31}(0)\ne 0$, all the focal conics 
for sufficiently small\/ $\tilde{s}$ are hyperbolas.
\end{corollary}

\begin{proof}
By Theorem \ref{thm:invexp} if $f_{31}(0) \ne 0$, 
then $a_{02}a_{20}$ diverges to $+\infty$
when $\tilde s\to0$.
Therefore, by Proposition \ref{prop:focalconic}, 
we obtain the assertion.
\end{proof}
As we remarked,
this fact can be
interpreted as a variant of the fact
about the singular curvature on cuspidal edges.

We give an example of the case $f_{31}(0)\ne0$ whose focal conics are ellipses when $\tilde{s}$ is not close to $0$ and are hyperbolas when $\tilde{s}$ is close to $0$.

\begin{example}\label{ex:ex4}
 Let $f$ be a deformation of an $S_1^+$ singularity defined by
 $$f(u,v,s)=(u,-u^2+v^2,u^2+v^3+vs^2+u^2v).$$
 In this case $f_{31}(0)\ne0$ holds. The focal conic of $f$ is an ellipse when $s < -1/4$, a parabola when $s= -1/4$, a hyperbola when $-1/4 <s < 0$ and  two transversal lines when $s=0$. See Figure \ref{fig:02}.
\end{example}

\begin{figure}[h!]
\centering
\includegraphics[width=0.6\linewidth, bb=0 0 1524 579]{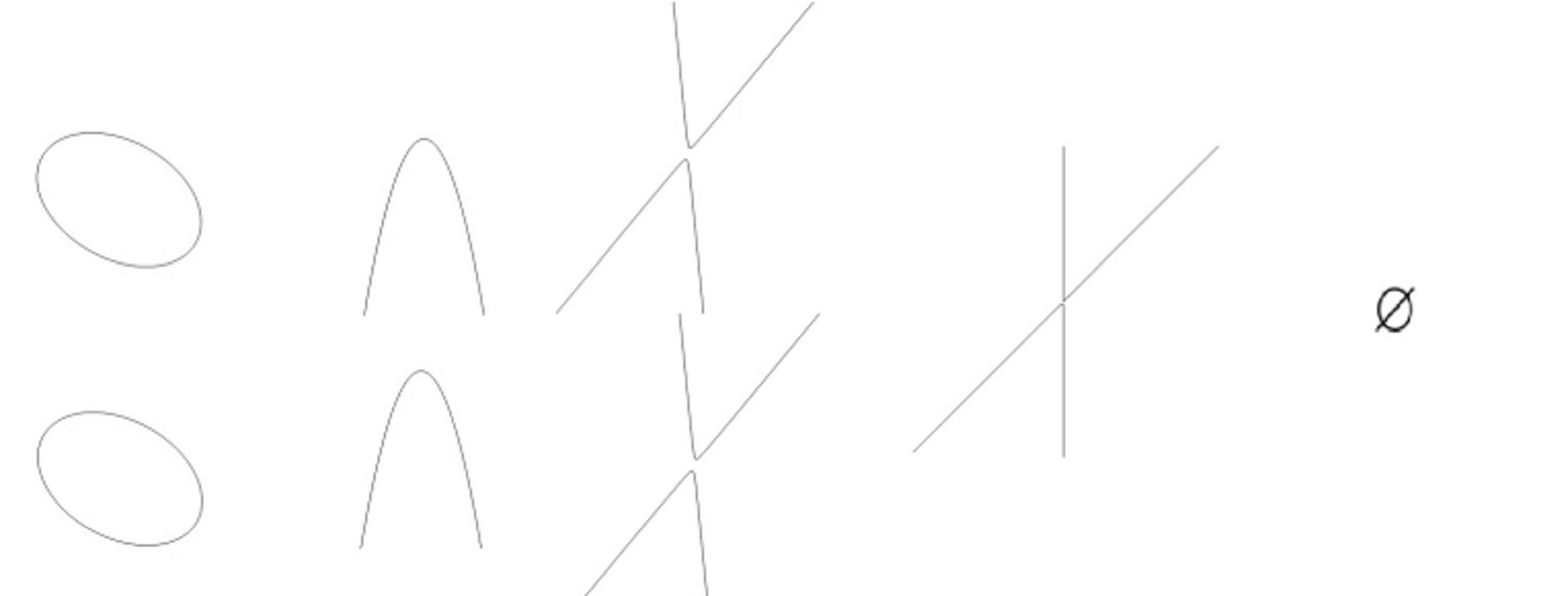}
\caption{The focal conics in Example \ref{ex:ex4} (from left to right $s=-1, -1/2, -1/5, 0, 1/5$)} 
\label{fig:02}
\end{figure}

Next we give an example in the case of $f_{31}(0)=0$. 
We can observe the focal conic is a parabola even if $\tilde{s}$ is close to $0$.

\begin{example}\label{ex:ex5}
 Let $f$ be a deformation of an $S_1^+$ singularity defined by
 $$f(u,v,s)=(u,v^2,v^3-vs^2+u^2v).$$
 In this case, $f_{31}(0)=0$ holds. The focal conic of $f$ is a parabola when $s<0$ and a line when $s=0$. See  Figure \ref{fig:03}.
\end{example}

\begin{figure}[h!]
\centering

\includegraphics[width=0.5\linewidth, bb=0 0 752 466]{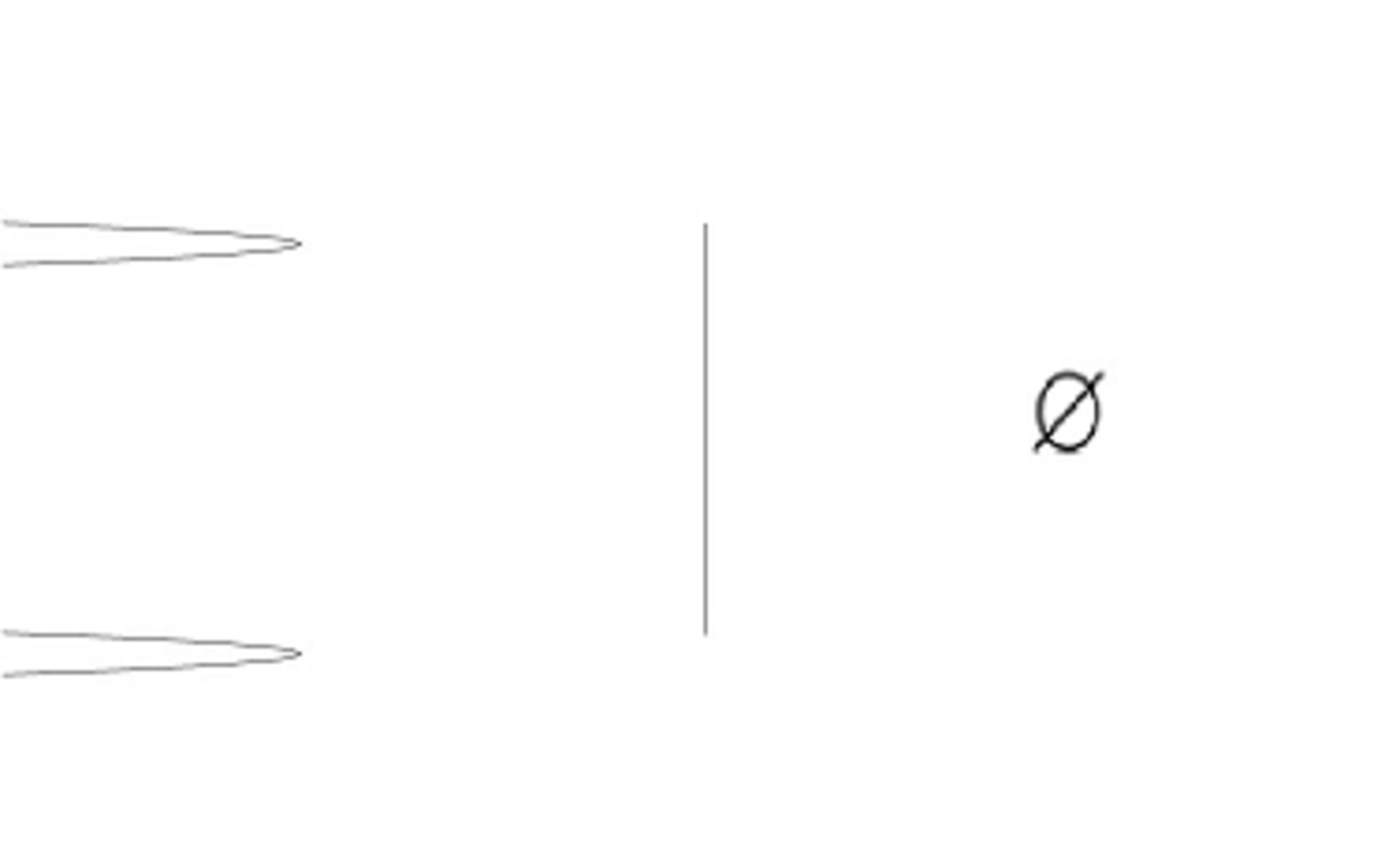}

\caption{The focal conics in Example \ref{ex:ex5} (from left to right $s=-1, 0, 1$)}
\label{fig:03}
\end{figure}

\subsection{Geometry on trajectories of singular points}
In this section, we give a geometric meaning for
the lowest order coefficients $f_{24}(0, 0)$ and $f_{34}(0, 0)$,
including the deformation parameters.
The trajectory of the singular points $S({f_{{\rm normal}}}^{-\tilde{s}^2})$ 
for the deformation of the $S_1^\pm$ singularities ${f_{{\rm normal}}}^s$ 
is a space curve passing through the origin. It is parameterized by
$$
\gamma (\tilde{s}):={f_{{\rm normal}}}^{-\tilde{s}^2}(u(\tilde{s}),0),
$$
where $u(\tilde{s})$ is given in Theorem \ref{thm:singcurve}.
Then the curvature $\kappa$ of $\gamma$ at $\tilde s=0$ is
$\kappa=2(f_{21}^2(0)+f_{31}^2(0))^{1/2}$,
which implies
$\kappa=(\kappa_u^2+\kappa_a^2)^{1/2}$.
Moreover, if $f_{21}^2(0)+f_{31}^2(0)\ne0$, then
$$f_{24}(0,0)=\dfrac{\tau f_{31}(0)-\dfrac{\kappa'c_{20}f_{21}(0)}{\kappa}+3\dfrac{df_{21}}{du}(0)}{3c_{20}^2}$$
and
$$f_{34}(0,0)=\dfrac{-\tau f_{21}(0)-\dfrac{\kappa'c_{20}f_{31}(0)}{\kappa}+3\dfrac{df_{31}}{du}(0)}{3c_{20}^2}$$
hold, where $\tau$ is the torsion of $\gamma$, and  
$\kappa'=d\kappa/d\tilde{s}$.

\begin{remark}
In Theorem \ref{thm:normal}, if we allow the isometry $T\in SO(3)$ to depend on the deformation
parameter $s$, written as $T(s)$, then one can reduce to $(f_{21}(0),f_{31}(0))=(0,0)$ in  ${f_{{\rm normal}}}^s$.
This can be realized by considering the following diagram:
$$
\begin{CD}
     (\R^2\times\R,0) @>{({f_{{\rm normal}}}^s,\operatorname{id})}>> (\R^3\times\R,0) \\
  @V{\phi}VV    @V{(T^{-1}(s),\phi_3(s))}VV \\
     (\R^2\times\R,0)   @>{(f,\operatorname{id})}>>  (\R^3\times\R,0)
     \end{CD}
$$
\end{remark}

\begin{acknowledgements}
The author would like to express her deepest gratitude to Kentaro Saji and  Wayne Rossman for fruitful discussions and advice.
\end{acknowledgements}


\medskip
{\footnotesize
\begin{flushright}
\begin{tabular}{l}
Department of Mathematics,\\
Graduate School of Science, \\
Kobe University, \\
1-1, Rokkodai, Nada, Kobe \\
657-8501, Japan\\
E-mail: {\tt 231s010s@stu.kobe-u.ac.jp}
\end{tabular}
\end{flushright}}

\end{document}